\documentclass{amsart}
\usepackage{graphicx,euscript,eufrak,amsmath,amsthm, color, mathrsfs,hyperref}
\usepackage[margin=1.3in]{geometry} 

\theoremstyle{definition}

\newtheorem{proposition}{Proposition}

\newcommand{\vp}{\varphi}

\newcommand{\ZZ}{\mathbb{Z}}

\newcommand{\ti}{\textit}

\long\def\ignore#1{}

\def\nd{\noindent}

\def\ti{\textit}

\def\RR{\mathbb{R}}
\def\ZZ{\mathbb{Z}}
\def\CC{\mathbb{C}}

\begin{document}

\title{A note on fourier eigenfunctions in four dimensions}

\author{Daniel Lautzenheiser}
\begin{abstract}
	In this note, we exhibit a weakly holomorphic modular form for use in constructing a Fourier eigenfunction in four dimensions. Such auxiliary functions may be of use to the D4 checkerboard lattice and the four dimensional sphere packing problem.
\end{abstract}
\address{Eastern Sierra College Center, 4090 W. Line St, Bishop, CA, 93514}
\thanks{\nd Date: \today.}

\email{daniel.lautzenheiser@cerrocoso.edu}

\maketitle

\section{introduction}

The sphere packing density is the proportion of $ \RR^d $ occupied by non-overlapping unit balls. In recent years, the sphere packing problem of finding the most dense arrangement of spheres in $ \RR^d $ has regained interest. This problem has been recently solved in 8 and 24 dimensions \cite{viazovska2017sphere, cohn2017sphere} by means of constructing a specialized radial Schwartz function using Fourier and complex analytic methods. The sphere packing problem has otherwise only been solved in dimensions 1,2, and 3 \cite{toth1943dichteste, hales2005proof}. A recent work \cite{cohn2022three} conjectures conditions for which the sphere packing problem in dimension 4 may be solved. In this paper, we present a function for possible use toward this conjecture.

If $ \Lambda $ is a lattice in $ \RR^d $ with minimal nonzero vector length $ \rho $, then a sphere packing associated to $ \Lambda $ may be defined by placing spheres of radius $ \rho/2 $ at each lattice point. In this case, there is one sphere for each copy of the lattice fundamental cell $ \RR^d/\Lambda $ and the sphere packing density is the ratio
\begin{align*}
\frac{\text{Vol}(B_{\rho/2}^d)}{\text{Vol}(\RR^d/\Lambda)}.
\end{align*}
Let $ f:\RR^d \rightarrow \CC $ be integrable with Fourier transform
\begin{align*}
\mathcal{F}(f)(\xi) = \widehat{f}(\xi) = \int_{\RR^d} f(x) e^{-2 \pi i x \cdot \xi} dx, \quad \xi\in{\RR^d}.
\end{align*}
If $ f: \RR^d \rightarrow \RR $ is a Schwartz function,  Poisson summation enforces 
\begin{align*}
\sum_{x\in{\Lambda}} f(x) = \frac{1}{\text{Vol}(\RR^d/\Lambda)} \sum_{y\in{\Lambda^*}} \widehat{f}(y)
\end{align*}
where $ \Lambda^* $ is the dual lattice. We can modify Poisson summation to form an inequality
\begin{align*}
f(0) \geq \sum_{x\in{\Lambda}} f(x) = \frac{1}{\text{Vol}(\RR^d/\Lambda)} \sum_{y\in{\Lambda^*}} \widehat{f}(y) \geq \frac{\widehat{f}(0)}{\text{Vol}(\RR^d/\Lambda)}
\end{align*}
if $ f $ is positive definite ($ \widehat{f} \geq 0 $) and $ f(x) \leq 0 $ for all $ ||x|| \geq \rho $. This yields $ \text{Vol}(\RR^d/\Lambda) \geq 1$ if $ \widehat{f}(0) = f(0) $. In light of the lattice sphere packing density, the existence of an auxiliary function $ f $ turning the above inequalities into equalities shows that the $ \Lambda $ lattice sphere packing density is $ \leq \text{Vol}(B_{\rho/2}^d) $. This argument, first put forth in \cite{cohn2003new}, converts the sphere packing problem into an analysis problem. To construct a \ti{magic function} $ f $ in which the above inequalities are tight ($ f $ and $ \widehat{f} $ also vanish on nonzero lattice points), it is typical to write $ f = f_+ + f_- $ where $ \widehat{f_+} = f_+ $ and $ \widehat{f_-} = -f_- $ and $ f_+,f_- $ have sign radius
\begin{align*}
r_f = \inf \{ M \geq 0: f(\{||x|| \geq M\}) \subseteq [0,\infty) \}
\end{align*}
equal to $ \rho $. Minimizing the quantity $ (r_fr_{\widehat{f}})^{1/2} $ is a type of uncertainty principle which is currently only solved in dimension $d=12$ when $\widehat{f}=f$ and dimensions $d\in{\{1,8,24\}}$ when $\widehat{f}=-f$, with the cases $d\in{\{8,24\}}$ corresponding to the $-1$ eigenfunctions constructed in the solving of the sphere packing problems in dimensions 8 and 24 \cite{logan1983extremalII, bourgain2010principe,gonccalves2017hermite,cohn2019optimal,gonccalves2020new}.

We construct a radial Schwartz function $ f_+:\RR^4 \rightarrow \RR $ such that $ \widehat{f_+} = f_+ $ with sign radius 
$ r_{f_+} = \sqrt{2}$, which is also the minimal vector length in the $ D_4 $ checkerboard lattice \cite{conway2013sphere}. $ f_+ $ is not sharp in the sense of minimizing the sign radius since numerical results in \cite{cohn2019optimal} demonstrate auxiliary $ f $ with $ r_f < 0.97 $. It is possible, in light of the slackness conditions imposed in conjecture 6.1 of \cite{cohn2022three} that $ f_+ $ may be of use toward the sphere packing problem in dimension 4.

To make this note more self-contained, we will include some of the details and arguments from \cite{viazovska2017sphere, cohn2017sphere, cohn2019optimal, rolen2020note, radchenko2019fourier}.

\section{a +1 eigenfunction in four dimensions}

We let $ q = e^{2\pi i z} $ with $ \mathrm{Im} (z) >0 $. Define the Jacobi theta series

\begin{align*}
&\Theta_2(z) = \sum_{n\in{\ZZ+\frac{1}{2}}}q^{\frac{1}{2}n^2}, \\
&\Theta_3(z) = \sum_{n\in{\ZZ}} q^{\frac{1}{2}n^2},  \\
& \Theta_4(z) = \sum_{n\in{\ZZ}}(-1)^nq^{\frac{1}{2}n^2}.
\end{align*}
Under the action of $ z \mapsto -1/z $, these theta functions satisfy
\begin{align*}
& \Theta_2(-1/z) = (-iz)^{1/2}\Theta_4(z), \\
& \Theta_3(-1/z) = (-iz)^{1/2} \Theta_3(z), \\
& \Theta_4(-1/z) = (-iz)^{1/2} \Theta_2(z).
\end{align*}
Under the action of $ z \mapsto z+1 $ they satisfy
\begin{align*}
&\Theta_2(z+1) = e^{i \pi/4}\Theta_2(z), \\
& \Theta_3(z+1) = \Theta_4(z), \\
& \Theta_4(z+1) = \Theta_3(z).
\end{align*}
The fourth powers $ \Theta_2^4 $, $ \Theta_3^4 $, $ \Theta_4^4 $ have the following leading terms at the cusp $ i \infty $:
\begin{align*}
& \Theta_2^4 = 16q^{1/2}+64q^{3/2} + O(q^{5/2}), \\
& \Theta_3^4 = 1+8q^{1/2}+24q + 32q^{3/2}+24q^2+O(q^{5/2}), \\
& \Theta_4^4 = 1-8q^{1/2}+24q-32q^{3/2}+24q^2 + O(q^{5/2}).
\end{align*} 
From these $ q $-expansions one can observe the \ti{Jacobi identity}: $ \Theta_3^4 = \Theta_2^4 + \Theta_4^4 $. Additionally, $ \Theta_2^4 $, $ \Theta_3^4 $, and $ \Theta_4^4 $ generate the ring of holomorphic modular forms of weight $ k=2 $ for the congruence subgroup $ \Gamma(2)$ \cite{bruinier20081}. We normalize the \ti{modular discriminant} $ \Delta $, which can be written in several ways: 
\begin{align*}
\Delta(z) =  q \prod_{n \geq 1}(1-q^n)^{24} = \eta(z)^{24}= \sum_{n \geq 1} \tau(n)q^n =  \frac{E_4^3 - E_6^2}{1728}.
\end{align*}
The function $ \Delta $ is a weight 12 cusp form for $ SL_2(\ZZ) $ \cite{murty2015problems}. A salient point for our purposes is that $ \Delta(it)>0 $ for $ t>0 $. This can be seen upon setting $z=it$ in the above product expansion.

For $ x\in{\RR^d} $, the Fourier transform of a Gaussian function is 
\begin{align*}
e^{\pi i ||x||^2 z} \rightarrow (-i z)^{-d/2}e^{\pi i ||x||^2(-1/z)}.
\end{align*}
Thus, a natural way to construct Fourier eigenfuctions is to work with linear combinations of Gaussians - the continuous version being the Laplace transform. For $ r>\sqrt{2} $, let
\begin{align*}
a(r) = -4i \sin(\pi r^2/2)^2 \int_0^\infty \vp(it)e^{-\pi r^2 t} dt 
\end{align*}
which may be interpreted as the Laplace transform of a weighting function $ g(t) = \vp(it) $ evaluated at $ \pi r^2 $,  multiplied by the root forcing function $ -4i\sin(\pi r^2/2)^2 $. We derive conditions on $ \vp $ that make $ \widehat{a} = a $. We may interpret $ a(r) = a(x) $ for $ x\in{\RR^4} $ with $ ||x||=r $ since we are interested in radial functions \cite{bourgain2010principe,cohn2003new}. Letting $ z=it $ rotates the  integration to the positive imaginary axis giving
	\begin{align*}
	a(r) &= \int_0^{i \infty}\vp(z)e^{\pi i r^2(z+1)}dz + \int_0^{i \infty} \vp(z) e^{\pi i r^2 (z-1)}dz -2 \int_0^{i \infty} \vp(z)e^{\pi i r^2z}dz \\
	& = \int_1^{1+i \infty}\vp(z-1)e^{\pi i r^2 z}dz + \int_{-1}^{-1+i \infty} \vp(z+1) e^{\pi i r^2 z} dz - 2\int_0^{i \infty} \vp(z) e^{\pi i r^2z}dz.
	\end{align*}
	Addressing the convergence of the integral at $ i \infty $ and near the real line, if $ |\vp(it)|=O(e^{2 \pi t}) $ as $ t \rightarrow \infty $ and $ |\vp(it)| = O(e^{-\pi/t}) $ as $ t \rightarrow 0^+ $, we may apply Cauchy's theorem and write $ \int_1^{1+i \infty} + \int_{i \infty}^i + \int_i^1 = 0 $ and $ \int_{-1}^{-1+i \infty} + \int_{i \infty}^i + \int_i^{-1} = 0 $. The path of integration is chosen to be perpendicular to the real line at the endpoints $1$ and $-1$. If also $ \vp(z+2) = \vp(z) $, then
	\begin{align*}
	a(r) & = 2\int_i^{i \infty} (\vp(z+1)-\vp(z))e^{\pi i r^2 z} dz + \int_1^i \vp(z-1)e^{\pi i r^2 z} dz \\& + \int_{-1}^i \vp(z+1)e^{\pi i r^2 z} dz - 2\int_0^i\vp(z)e^{\pi i r^2z}dz.
	\end{align*}
	Taking the Fourier transform and maintaining the variable $ r $ (since the Fourier transform of a radial function is radial),  substitute $ \omega = \frac{-1}{z} $: 
	\begin{align*}
	\widehat{a}(r)  =  2 & \int_i^0(\vp(\tfrac{-1}{\omega}+1)-\vp(\tfrac{-1}{\omega}))i^{-d/2}\omega^{d/2-2}e^{\pi i r^2 \omega} d \omega \\
	 + & \int_{-1}^i \vp(\tfrac{-1}{\omega}-1)i^{-d/2}\omega^{d/2-2}e^{\pi i r^2 \omega} d \omega \\
	 + & \int_1^i \vp(\tfrac{-1}{\omega}+1)i^{-d/2}\omega^{d/2-2}e^{\pi i r^2 \omega} d \omega \\
	 -2 & \int_{i \infty}^i \vp(\tfrac{-1}{\omega})i^{-d/2}\omega^{d/2-2}e^{\pi i r^2 \omega} d \omega.
	\end{align*}
	Note that we are exchanging the Fourier transform with the contour integrals. The equality $ \widehat{a} = a $ can be achieved if equality happens at the level of the integrand functions. Comparing the terms in the integral $ \int_i^{i \infty} $ results in
	\begin{align}\label{functional equation (main) for vp}
	\vp(z+1) - \vp(z) = \vp(\tfrac{-1}{z})i^{-d/2}z^{d/2-2}.
	\end{align}
	Comparing terms in $ \int_0^i $ results in $ \vp(z) = (\vp(\tfrac{-1}{z}+1)-\vp(\tfrac{-1}{z}))i^{-d/2}z^{d/2-2} $, which is the same as (\ref{functional equation (main) for vp}) after inverting back $ z \mapsto \frac{-1}{z} $. Comparing terms in $ \int_1^i $, we get
	\begin{align}\label{functional eqn vp modified}
	\vp(z+1) = \vp(\tfrac{-1}{z}+1)i^{-d/2}z^{d/2-2}.
	\end{align}
	From (\ref{functional equation (main) for vp}), 
	\begin{align*}
	\vp(z+1) & = \vp(z) + \vp(\tfrac{-1}{z})i^{-d/2}z^{d/2-2} \\
	& = (\vp(\tfrac{-1}{z}+1)-\vp(\tfrac{-1}{z}))i^{-d/2}z^{d/2-2}+\vp(\tfrac{-1}{z})i^{-d/2}z^{d/2-2} \\
	& = \vp(\tfrac{-1}{z}+1)i^{-d/2}z^{d/2-2}.
	\end{align*}
	Thus (\ref{functional eqn vp modified}) follows from (\ref{functional equation (main) for vp}). The comparison in $ \int_{-1}^i $ is similar.

\begin{proposition}
	Let $ d=4 $ and let
	\begin{align}
	\vp = \frac{\Theta_4^{12}(\Theta_3^{12}+\Theta_2^{12})}{\Delta}.
	\end{align}
	Then, $ f_+(r) = i a(r) $ is a radial Schwartz function invariant under the Fourier transform and $ f_+ $ has sign radius $ r_{f_+} = \sqrt{2} $.
\end{proposition}

\begin{proof}
	Based on the functional equations for $ \Theta_2,\Theta_3, \Theta_4 $ and $ \Delta $, 
	\begin{align*}
	\vp(z+1) = \frac{\Theta_3^{12}(\Theta_4^{12}-\Theta_2^{12})}{\Delta}
	\end{align*}
	and
	\begin{align*}
	\vp(-1/z) = \frac{-z^6\Theta_2(z)^{12}(-z^6\Theta_3(z)^{12}-z^6\Theta_4(z)^{12})}{z^{12}\Delta(z)} = \frac{\Theta_2^{12}(\Theta_3^{12}+\Theta_4^{12})}{\Delta}.
	\end{align*}
	Thus, $ \vp $ satisfies (\ref{functional equation (main) for vp}) for $ d=4 $.
    To show that $\vp$ is a weight 0 weakly holomorphic modular form for the congruence subgroup $\Gamma(2)$, written $ \vp \in{M_0^!(\Gamma(2))}$, it suffices to check that $ \vp(\gamma z) = \vp(z) $ for a set of generators of $ \Gamma(2)$. $ \Gamma(2) $ is generated by $ \left(
	\begin{smallmatrix}
	1&2\\0&1\\
	\end{smallmatrix}\right) $, $ \left(
	\begin{smallmatrix}
	1&0\\2&1\\
	\end{smallmatrix}
	\right) $, and $ -I $. Observe that
	$ \vp\left( \left(
	\begin{smallmatrix}
	1&2\\0&1
	\end{smallmatrix}
	\right)z\right) =  \vp(z+2) = \vp(z) $ directly from the definitions of $\Delta, \Theta_2,\Theta_3$, and $\Theta_4 $. Using (\ref{functional equation (main) for vp}), 
	\begin{align*}
	\vp \left(
	\begin{pmatrix}
	1&0\\2&1\\
	\end{pmatrix}
	z \right) & = \vp\left(\frac{z}{2z+1}\right) = \vp\left(\frac{-1}{-1/z-2}\right) \\
	& = \vp \left( \frac{-1}{z}-2 \right) - \vp \left( \frac{-1}{z}-2+1 \right) \\
	& = \vp\left( \frac{-1}{z} \right) - \left( \vp \left( \frac{-1}{z} \right) - \vp\left(\frac{-1}{-1/z} \right)\right) = \vp(z).
	\end{align*}
	The first few Fourier coefficients of $ \vp $ are
	\begin{align}\label{q series / fourier expansion for vp}
	& \vp(z) = q^{-1}-24+4096q^{1/2}-98028q+O(q^{3/2}), \\
	& \vp(-1/z) = 8192q^{1/2} + O(q^{3/2}). \nonumber 
	\end{align}
	The asymptotics and the fact that $ f_+ $ is a Schwartz function follow directly from the growth rate of these Fourier coefficients since $ \vp $ is a weakly holomorphic modular form \cite{bruinier2002borcherds}. $ f_+ $ has sign radius $ \sqrt{2} $ since we may observe that for $ t>0 $, $ \Delta(it) > 0  $ and each $ \Theta_j(it) \in{\RR} $. Thus, $ f_+(r) \geq 0 $ for $ r > \sqrt{2} $.
	
\end{proof}

Based on the $ q $-series expansion (\ref{q series / fourier expansion for vp}), it is possible to write

\begin{align*}
a(r) & = -4i\sin(\pi r^2/2)^2 \int_0^\infty (e^{2\pi t}-24+O(e^{-\pi t}))e^{-\pi r^2 t}dt \\
& = -4i \sin(\pi r^2/2)^2 \left(\frac{1}{\pi(r^2-2)} - \frac{24}{\pi r^2} + \int_0^\infty (\vp(it)-e^{2\pi t}+24)e^{-\pi r^2 t}dt \right) \\
\end{align*}
which converges for all $ r \geq 0 $. In particular, $ f_+(0) = 0 $, so $ f_+ $ is not positive definite.

\section{a -1 eigenfunction and concluding remarks}

To construct a $ -1 $ eigenfunction, it is possible to begin with a similar function shape
\begin{align*}
b(r) = -4 \sin(\pi r^2/2)^2 \int_0^{i \infty} \phi(-1/z)z^{d/2-2}e^{\pi i r^2z}dz.
\end{align*}
Comparing integrand terms as before, we get
\begin{align}\label{quasi modular property}
\phi\left(\frac{-1}{z-1}\right)(z-1)^{d/2-2} + \phi \left( \frac{-1}{z+1} \right)(z+1)^{d/2-2} -2\phi\left(\frac{-1}{z}\right)z^{d/2-2} = 2 \phi(z).
\end{align}
So, the form given for $ b(r) $ enforces (\ref{quasi modular property}) when $ \widehat{b} = -b $. Using the Eisenstein series $ E_2(z) = 1-24 \sum_{n\geq 1} \sigma_1(n)q^n $ (which is usually introduced as a basic example of a quasi-modular form) and the transformation property
\begin{align*}
E_2\left(\frac{-1}{z}\right) = z^2E_2(z) + \frac{6z}{\pi i},
\end{align*}
a straightforward calculation shows that $ E_2 $ satisfies (\ref{quasi modular property}) for $ d=4 $. Thus, setting $ \phi = E_2 $ above gives a $ -1 $ eigenfunction also with sign radius $ \sqrt{2} $. So, $ b(r) $ provides a simple construction of an eigenfunction with the basic depth 1 Eisenstein series $ E_2 $ within the integral. As before, $ b(r) $ is not sharp since the numerical bounds from table 4.1 in \cite{cohn2019optimal} give $ A_-(4) \leq 1.204 < \sqrt{2} $. Furthermore, using the checkerboard lattice $ D_4 $ from \cite{conway2013sphere} with minimal radius $ \sqrt{2} $, the dual lattice $ D_4^* $ is homothetic to $ D_4 $ but will have minimal radius 1. We can scale $ D_4 $ so that $ D_4 $ and $D_4^* $ both have minimal radius $ 2^{1/4} $, however constructing eigenfunctions $ f:\RR^4 \rightarrow \RR $ of the above shape with sign radius $ 2^{1/4} $ appears to remain a difficult problem. 



\newpage
\bibliographystyle{plain}

\bibliography{references}

\begin{thebibliography}{10}

\bibitem{bourgain2010principe}
Jean Bourgain, Laurent Clozel, and Jean-Pierre Kahane.
\newblock Principe d’heisenberg et fonctions positives.
\newblock In {\em Annales de l'institut Fourier}, volume~60, pages 1215--1232,
  2010.

\bibitem{bruinier2002borcherds}
Jan~H Bruinier.
\newblock {\em Borcherds products on O (2, l) and Chern classes of Heegner
  divisors}.
\newblock Number 1780. Springer Science \& Business Media, 2002.

\bibitem{bruinier20081}
Jan~Hendrik Bruinier, Gerard Van~der Geer, G{\"u}nter Harder, and Don Zagier.
\newblock {\em The 1-2-3 of modular forms: lectures at a summer school in
  Nordfjordeid, Norway}.
\newblock Springer Science \& Business Media, 2008.

\bibitem{cohn2022three}
Henry Cohn, David de~Laat, and Andrew Salmon.
\newblock Three-point bounds for sphere packing.
\newblock {\em arXiv preprint arXiv:2206.15373}, 2022.

\bibitem{cohn2003new}
Henry Cohn and Noam Elkies.
\newblock New upper bounds on sphere packings i.
\newblock {\em Annals of Mathematics}, pages 689--714, 2003.

\bibitem{cohn2019optimal}
Henry Cohn and Felipe Gon{\c{c}}alves.
\newblock An optimal uncertainty principle in twelve dimensions via modular
  forms.
\newblock {\em Inventiones mathematicae}, 217(3):799--831, 2019.

\bibitem{cohn2017sphere}
Henry Cohn, Abhinav Kumar, Stephen Miller, Danylo Radchenko, and Maryna
  Viazovska.
\newblock The sphere packing problem in dimension $24$.
\newblock {\em Annals of Mathematics}, 185(3):1017--1033, 2017.

\bibitem{conway2013sphere}
John~Horton Conway and Neil James~Alexander Sloane.
\newblock {\em Sphere packings, lattices and groups}, volume 290.
\newblock Springer Science \& Business Media, 2013.

\bibitem{gonccalves2017hermite}
Felipe Gon{\c{c}}alves, Diogo~Oliveira e~Silva, and Stefan Steinerberger.
\newblock Hermite polynomials, linear flows on the torus, and an uncertainty
  principle for roots.
\newblock {\em Journal of Mathematical Analysis and Applications},
  451(2):678--711, 2017.

\bibitem{gonccalves2020new}
Felipe Gon{\c{c}}alves, Jo{\~a}o~PG Ramos, et~al.
\newblock New sign uncertainty principles.
\newblock {\em arXiv preprint arXiv:2003.10771}, 2020.

\bibitem{hales2005proof}
Thomas~C Hales.
\newblock A proof of the kepler conjecture.
\newblock {\em Annals of mathematics}, pages 1065--1185, 2005.

\bibitem{logan1983extremalII}
BF~Logan.
\newblock Extremal problems for positive-definite bandlimited functions. ii.
  eventually negative functions.
\newblock {\em SIAM Journal on Mathematical Analysis}, 14(2):253--257, 1983.

\bibitem{murty2015problems}
M~Ram Murty, Michael Dewar, and Hester Graves.
\newblock {\em Problems in the theory of modular forms}.
\newblock Springer, 2015.

\bibitem{radchenko2019fourier}
Danylo Radchenko and Maryna Viazovska.
\newblock Fourier interpolation on the real line.
\newblock {\em Publications math{\'e}matiques de l'IH{\'E}S}, 129(1):51--81,
  2019.

\bibitem{rolen2020note}
Larry Rolen and Ian Wagner.
\newblock A note on schwartz functions and modular forms.
\newblock {\em Archiv der Mathematik}, 115(1):35--51, 2020.

\bibitem{toth1943dichteste}
L~Fejes T{\'o}th.
\newblock {\"U}ber die dichteste kugellagerung.
\newblock {\em Math. Z}, 48(676-684), 1943.

\bibitem{viazovska2017sphere}
Maryna~S Viazovska.
\newblock The sphere packing problem in dimension 8.
\newblock {\em Annals of Mathematics}, pages 991--1015, 2017.

\end{thebibliography}

\end{document}